\documentclass[12pt,a4paper]{article}
\usepackage{amssymb}
\usepackage{amsmath,amsfonts,amsthm}

\usepackage{graphicx}
\usepackage{amsmath}

\usepackage{amsfonts}
\usepackage{amsthm,amscd}
\usepackage{graphicx}
\usepackage{amsmath}
\usepackage{amssymb}
\usepackage{amstext}

\newtheorem{theorem}{Theorem}
\newtheorem{corollary}[theorem]{Corollary}
\newtheorem{definition}[theorem]{Definition}
\newtheorem{example}[theorem]{Example}
\newtheorem{lemma}[theorem]{Lemma}

\newtheorem{proposition}[theorem]{Proposition}
\newtheorem{remark}[theorem]{Remark}

\title{On the selection of subaction and  measure for a subclass of potentials defined by P. Walters}

\author{A. T. Baraviera\thanks{baravi@mat.ufrgs.br, Instituto de Matem\'atica - UFRGS, Partially supported by DynEurBraz, PRONEX -- Sistemas
Dinamicos,  INCT, Convenio Brasil-Franca},
A. O. Lopes\thanks{arturoscar.lopes@gmail.com, Instituto de Matem\'atica - UFRGS - Partially supported by DynEurBraz, CNPq, PRONEX -- Sistemas
Dinamicos,  INCT, Convenio Brasil-Franca} and J. K. Mengue\thanks{jairokras@gmail.com,  Instituto de Matem\'atica - UFRGS,   CNPq Pos-doc Schollarship} }

\begin{document}

\maketitle

\begin{abstract}

Suppose $\sigma$ is the shift  acting on Bernoulli space $X=\{0,1\}^\mathbb{N}$,  and, consider a fixed function $f:X \to \mathbb{R}$, under the Waters's conditions (defined in a paper in ETDS 2007). For each real value $t\geq 0$ we consider the Ruelle Operator $L_{tf}$. We are interested in the main eigenfunction $h_t$ of $L_{tf}$, and, the main eigenmeasure $\nu_t$, for the dual operator $L_{tf}^*$, which we consider normalized in such way  $h_t(0^\infty)=1$, and, $\int h_t \,d\,\nu_t=1, \forall t>0$.  We denote $\mu_t= h_t \nu_t$ the Gibbs state for the potential $t\, f$. By selection of a subaction $V$, when the temperature goes to zero (or, $t\to \infty$), we mean the existence of the limit
$$V:=\lim_{t\to\infty}\frac{1}{t}\log(h_{t}).$$

By selection of a measure $\mu$, when the temperature goes to zero (or, $t\to \infty$), we mean the existence of the limit (in the weak$^*$ sense)
$$\mu:=\lim_{t\to\infty} \mu_t.$$

We present a large family of non-trivial examples of $f$ where the selection of measure exists. These $f$  belong to a sub-class  of potentials introduced by P. Walters. In this case, explicit expressions for the selected $V$ can be obtained for a certain large family of parameters.


\end{abstract}

\section{Introduction}

We consider $X=\{0,1\}^{\mathbb{N}}$ with the usual metric
$$d(x,y) =\theta^{N},\ \ x_{1}=y_{1},...,x_{N-1}=y_{N-1},x_{N}\neq y_{N},$$
where $ x=(x_{1}x_{2}...), y=(y_{1}y_{2}...),$
with $\theta$ fixed $0<\theta<1$.

 We also consider here a fixed potential (just a function)  $f:X \to \mathbb{R}$,  which satisfies Walter's summable condition \cite{W} (to be defined later). There are Holders functions among the set of such potentials. For each real value $t\geq 0$, we denote  $L_{tf}$ the Ruelle Operator given by the potential $t\, f$; that is,  $L_{tf}$, acting on continuous functions $w:X\to\mathbb{R}$, is defined by
\[L_{tf}(w)(x) := \sum_{\sigma(y)=x}e^{tf(y)}w(y).\]
The parameter $t=\frac{1}{T}>0$ represents the inverse of the temperature $T$ on Thermodynamic Formalism \cite{PP} \cite{Kel}.
In the setting of Statistical Mechanics we are considering a problem in the lattice $\mathbb{N}$. The potential $t\, f$ represents an interaction on the lattice. When $f$ is a Holder function, any statistical problem in the lattice $\mathbb{Z}$ can be transformed into a problem in the lattice $\mathbb{N}$ \cite{PP}.

 We are interested in the main eigenfunction $h_t$ of $L_{tf}$, and, the main eigenmeasure $\nu_t$, for the dual operator $L_{tf}^*$ (both with the same eigenvalue).

  {\bf  For the purpose of selection of subaction, we assume here that the eigenfunction $h_t$ satisfies, for all $t$, the expression $h_t(0^\infty)=1$.}

 We assume that the eigen-measure $\nu_t$
 satisfies the normalization $\int h_t \,d\,\nu_t=1$.

We denote by ${\cal M}$ the set of invariant probabilities for the shift.

The probability $\mu_t = h_t \, \nu_t$ belongs to ${\cal M}$, and, it is called the Gibbs state for the potential $t \, f$.


The potentials that we consider here are such that the Gibbs state is  an equilibrium probability for $f$, that means, such $\mu_t$ maximizes Pressure (or, Free Energy):
\[P(tf):= \sup_{\rho \in \cal M} \left(h(\rho) + \int t\, f \, d \rho\right) = h(\mu_{t}) + \int t\,f\, d\mu_{t} ,\]
 where $h(\rho)$ is the Kolmogorov entropy of $\rho$ \cite{PP}.

A probability $\hat{\rho}$ which maximizes
$$ \int \, f \, d\, \rho,  \,\,\rho \in {\cal M},$$
is called a maximizing probability for $f$ \cite{CoG} \cite{CLT} \cite{Jenkinson}  \cite{Bousch1} \cite{leplaideur-max} \cite{Bousch-walters} \cite{Mo} \cite{BCLMS}. We denote $\beta(f)$ this maximal value.

One can show for the potentials we consider here that
any accumulation point of $\mu_{t_i}$ $t_i \to \infty,$ is a maximizing probability for $f$ (due to the maximal pressure property the proof is the same as in \cite{CoG}).

Even for Holder potentials $f$, this maximizing probability is, of course, not necessarily  unique. In general, it  is not  positive on all open sets (even when this maximizing probability is unique), which is quite different of the case when one considers the Gibbs state $\mu_t$ at non zero temperature, that is, $t\neq \infty$ (in the case $f$ is Holder) \cite{CLT}.
 If the maximizing probability is unique, then, of course, the limit $\mu_t$ exist, and, it is equal to this probability. In this case, we have a trivial example of selection of limit measure, when $t \to \infty$.

 The selection problem is nontrivial only when there is more than one maximizing probability. There exists an interesting example by R. Chazottes and M. Hochman  \cite{chazottes-cv} of a Holder potential $f$ such that $\mu_t$ does not converges, when $t\to \infty$.
 If the potential depends on a finite number of coordinates, then the limit exists \cite{Bremont} \cite{leplaideur-max}. The selected measure do not have to be ergodic (see also an example in the end of the paper).

 An interesting problem is to find sufficient conditions to assure  the existence of a limit probability in the case the maximizing probability is not unique. And, also, to find explicit examples where one can exhibit the limit selected $V$. To address these questions is one the main purposes of the present paper. Part of our work is inspired by results on selection from the paper
\cite{BLL} (a different setting where it is considered a Bernoulii space with three symbols).

This kind of question -  existence, or non-existence of limit probability at temperature zero -  it is very important in Statistical Mechanics (see \cite{VE}, for the case where the state space is $S^1$, not $\{0,1\}$, the so called $XY$ Model, where the authors show an example with no limit for the Gibbs probabilities at temperature $T$, when $T=\frac{1}{t}\to 0$).
We refer the reader to \cite{BCLMS} for other  related results on the general $XY$ model.

We are interested in potentials $f$ that depends on infinite coordinates.

\begin{definition}
We say that a continuous function $u:X \to \mathbb{R}$
is a calibrated subaction for $f$, if for any  $y$ in $X$, we have
$$
 u(y) =\sup_{\sigma(x)=y}\{f(x)+ u(x)-\beta(f)\}.
$$
\end{definition}

If we add a constant to a subaction we get a new subaction. If $f$ is a Holder potential and the maximizing probability is unique, then there is a unique subaction up to an additive constant \cite{CLT} \cite{GL1}.
If, for a given subsequence we have $\frac{1}{t_i}\log(h_{t_i})\to V_1$, then, $V_1$ is a calibrated subaction \cite{BLT} \cite{GL1}.

If the maximizing probability is unique and $f$ is Holder, then, under the normalizing condition $h_t (0^\infty)=1$, we have that $\frac{1}{t}\log(h_{t})$
 converges \cite{CoG} \cite{CLT}.

 Moreover, under the normalization $\int d \nu_t=1$, $\int h_t \, \nu_t=1$,  we  have that $\frac{1}{t}\log(h_{t})$ is equicontinuous.
If the maximizing probability is unique,  $f$ is Holder, and we consider two (possible) different convergent subsequence
$\frac{1}{t_i}\log(h_{t_i})\to V_1$, and $\frac{1}{s_i}\log(h_{s_i})\to V_2$, then, of course,  $V_1-V_2$ is a constant \cite{BLT} \cite{GL1}.
In a forthcoming paper we will show that, if the maximizing probability is unique, then, with this normalization,  there exists also the limit $\frac{1}{t}\log(h_{t})$, as $t\to \infty$.

Under the uniqueness assumption of the maximizing probability, if we fix a point $x_0\in X$, and normalize $h_t$ by $h_t(x_0)=1$, for all $t$, then, the reasoning is easier, and $\frac{1}{t}\log(h_{t})$ converges always to a certain subaction $V$ (which satisfies $V(x_0)=0$), as $t\to \infty$.

Anyway, as we said before, here we are mainly interested in the case the maximizing probability is not unique. In this case there exists subactions which do not differ by a constant (see Theorems 10 and 12 in  \cite{GL1})

 We denote by $\sigma$ the shift operator acting on $X$. For $n\in \mathbb{N}$, we denote by $0^{n}$ and $1^{n}$ the points in $\{0,1\}^{n}$ given by
\[0^{n}:=\underbrace{0...0}_{n}, \ \ 1^{n}:=\underbrace{1...1}_{n} \,.\, \]

Moreover, given a finite sequence $x_{1}...x_{n} \in \{0,1\}^{n}$, we denote by $(x_{1}...x_{n})^{\infty}$ the periodic orbit in $X$ given by the successive repetition of $x_{1}...x_{n}$.

 Consider the class $R(X)$ (see \cite{W}) of function $f:X\to \mathbb{R}$ given by the rule: there exists sequences $\{a_{n}\}_{n\geq2}\to a$, $\{b_{n}\}_{n\geq 1} \to b$, $\{c_{n}\}_{n\geq2}\to c$, $\{d_{n}\}_{n\geq 1} \to d$, such that, for any $z\in X$, $p\geq 2$, $q\geq 1$:
\[f(0^{p}1z)=a_{p}, \, f(01^{q}0z) = b_{q}, \, f(1^{p}0z) = c_{p}, \, f(10^{q}1z)=d_{q}, \]
\[ f(0^{\infty})=a, \, f(01^{\infty})=b, \, f(1^{\infty})=c, \, f(10^{\infty})=d.\]

The potentials consider by F. Hofbauer \cite{Hof} \cite{L1} \cite{FL} (not Holder) are on this class $R(X)$. He was interested in phase transitions (at positive temperature).

Depending of the velocity of convergence of the sequences $a_p, b_p,c_p$ and $d_p$,  such function $f$ will be Holder (Lipschitz), or not. We call here \textbf{Walters's summable potential} (called $W(X,T)$ in \cite{W})  any function $f\in R(X)$ as above, such that $\sum_{n\geq 2}(a_{n}-a)$ and $\sum_{n\geq 2}(c_{n}-c)$ are convergent series (see \cite{W}, theorem 1.1 (ii)). This happen when $f$ is Holder.

 We point out that in some papers it is usual the terminology ``a potential satisfying the Walters's condition" (see \cite{Bousch-walters}), but this is a different concept from the one we use here (we consider in the sense of a class described in \cite{W}).

P. Walters \cite{W} introduce the class $R(X)$
and analyze for a certain subclass of potentials (see
Theorem 2.4
in page 1332) the Ruelle operator for these
potentials (not necessarily Holder).
In this case the main eigenvalue $\lambda_t$ is
given in an implicit expression
(see expression on top of page 1341 \cite{W},
or, in the beginning of our section 2), and,
in general, one can  not get it in an explicit form.
From the expression of this eigenvalue $\lambda_t$,
then, you have indeed the explicit
form of the eigenfunction $h_t$, and, the
eigenmeasure $\nu_t$ (see page 1341 \cite{W}).

For the class of summable Walters's potentials we consider all the above apply.

We point out that when $f$ is Holder
there is no phase transitions at non-zero
temperature (see \cite{PP}).
\bigskip

{\bf We will use bellow and in the rest of the paper the notation: for fixed $q\geq 1$, when $j=0$, we have that $a_{q+1}+...+a_{q+j}=0$.}

\bigskip

\bigskip
For the family of Walters's summable potentials, as we said,  we will analyze the two selection problems on the zero temperature limit. We will first prove the following:

\begin{theorem}\label{teo1}
Let $f$ be a Walters's summable potential. Suppose $a=c=\beta(f)$ and any maximizing measure is of the form  $s\,\delta_{0^{\infty}} + (1-s)\,\delta_{1^{\infty}}, \ \ s \in [0,1]$.
Then there exists a unique real value $A\leq 0$ such that
\footnotesize
\begin{align*}
2\beta(f) =& \max\{ \sup_{j\geq 0}(d_{1+j} + a_{2}+...+a_{1+j} - j\beta(f)), d+ \sum_{j=1}^{\infty}(a_{1+j}-\beta(f)) -A\} \nonumber \\
& + \max\{\sup_{j\geq 0}(b_{1+j} + c_{2}+...+c_{1+j} - j\beta(f)), b+ \sum_{j=1}^{\infty}(c_{1+j}-\beta(f)) -A\}.
\end{align*}
\normalsize
Also, there exists $V:=\lim_{t\to\infty}\frac{1}{t}\log(h_{t})$, and $V$ satisfies
\[V(0^{\infty})=0,\]
\[V(1^{\infty}) =  b -d -\beta(f) + \max\{\sup_{j}(d_{1+j} + a_{2}+...+a_{1+j} - j\beta(f)), d+ \sum_{j=1}^{\infty}(a_{1+j}-\beta(f)) -A\},\]
\[V(0^{q}1z) =  A -d  +\max\{\sup_{j}(d_{q+j} + a_{q+1}+...+a_{q+j} - j\beta(f)), d+ \sum_{j=1}^{\infty}(a_{q+j}-\beta(f)) -A\},\]
\begin{align*}
 V(1^{q}0z) = &  -d -\beta(f) +A \\
& \, + \max\{\sup_{j}(d_{1+j} + a_{2}+...+a_{1+j} - j\beta(f)), d+ \sum_{j=1}^{\infty}(a_{1+j}-\beta(f)) -A\}\\
& \, + \max\{\sup_{j}(b_{q+j} + c_{q+1}+...+c_{q+j} - j\beta(f)), b+ \sum_{j=1}^{\infty}(c_{q+j}-\beta(f)) -A\}.
\end{align*}

\end{theorem}

The value $A$ can be obtained in an explicit form (see Lemma \ref{epsilont}).
Clearly this function $V$ is also in $R(X)$.  It will be Holder if $f$ is Holder \cite{CLT}.

The second problem is: what we can say about selection of measures for the above class of potentials? We are going to determine the expressions of $\mu_{t}([0])$ and $\mu_{t}([1])$ (and, also for other cylinders), and,  later we will prove the following:

\begin{theorem}\label{teo2}
Let $f$ be a Walters's summable potential such that $a=c=0$, $b_{n}\equiv b$, $d_{n}\equiv d$, and assume the numbers $b,d,a_{n},c_{n}$ are strictly smaller than zero. If  $\sum_{j=2}^{\infty}a_{j} <b+d +  \sum_{j=2}^{\infty}c_{j}$, then $\mu_{t}\to \delta_{1^{\infty}}$ weakly*.
Also, $\frac{\mu_{t}([0])}{\mu_{t}([1])}\to 0$, exponentially fast.
\end{theorem}

Clearly an analogous result is true if $\sum_{j=2}^{\infty}c_{j} < b+d +  \sum_{j=2}^{\infty}a_{j}$. This result is related to the question considered by  Baraviera-Lopes-Leplaideur (\cite{BLL}). This show that the selection isn't a result only of the particular velocity of convergence of $a_{n}\to 0$ and $b_{n}\to 0$. In fact, the terms $a_{3},a_{4}...$ and $c_{3},c_{4},...$ aren't sufficiently to determine the selection of measure, because taking $a_{2}<<0$ or $c_{2}<<0$, we can get either case of inequality we want. For given $a=c=0$, $a_{3},a_{4},...<0$ and $c_{3},c_{4},...<0$, it is possible determine $b_{n},d_{n},a_{2},c_{2}$, in such way that $\mu_{t}\to \delta_{0^{\infty}}$ weakly*, and, $b'_{n},d'_{n},a'_{2},c'_{2}$, such that, $\mu_{t}\to \delta_{1^{\infty}}$ weakly*.

The exponential velocity of convergence given by Theorem $\ref{teo2}$, which also appears  in $(\cite{BLL})$, has applications in a better understanding of  Large Deviation Theory. It is known  that, if $f$ is Holder and has a unique maximizing measure $\mu_{\infty}$, then two calibrated subactions differ by a constant. In this case, from \cite{BLT}, the function $R_{+}:= \mu_{\infty}(f) + V\circ\sigma - V - f$, where $V$ is a calibrated subaction for $f$, is such that for any cylinder $K$:
\[\lim_{t\to\infty}\frac{1}{t}\log(\mu_{t}(K)) = -\inf_{x\in K}\sum_{j=0}^{\infty}R_{+}(\sigma^{j}(x)).\]

We call $I(x) = \sum_{j=0}^{\infty}R_{+}(\sigma^{j}(x))$ the deviation function.
It is a non-negative lower semi-continuous function which can attains in some points the value $\infty$. Moreover, $I$ is zero on the support of the maximizing probability.

It's natural to suppose that (even in the case the maximizing probability is not unique), if, it happens that there exists the limit $\frac{1}{t}\log(h_{t}) \to V$, then, defining $R_{+}:= \mu_{\infty}(f) + V\circ\sigma - V - f$, (where $\mu_{\infty}$ is any maximizing measure), we get that the above result should be also true ($\cite{M}$ proposition 47). Unfortunately,   this is false. Indeed, considering $f$ Holder, under the hypothesis of Theorem $\ref{teo2}$, we have
\[-\inf_{x\in [0]}\sum_{j=0}^{\infty}R_{+}(\sigma^{j}(x)) = -\sum_{j=0}^{\infty}R_{+}(\sigma^{j}(0^{\infty})) =0.\]
By the other hand by Theorem \ref{teo2}
\[\limsup_{t\to\infty}\frac{1}{t}\log(\mu_{t}([0])) = \limsup_{t\to\infty}\frac{1}{t}\log(\frac{\mu_{t}([0])}{\mu_{t}([1])}) <0.\]

 This means that if you want to estimate
$ \lim_{t\to\infty}\frac{1}{t}\log(\mu_{t}(K))$, for a general cylinder $K$,
the expression of the deviation function $I$, one can get from \cite{BLT} (under uniqueness assumption), does not apply to the case where there is no uniqueness of the maximizing probability (see considerations in section 3 in \cite{M}). The bottom line is: the explicit expression for the deviation function in this case should be different.

\section{Analysis of the eigenfunction and selection of subaction}

We start considering the case of positive temperature. Let $f$ a summable Walters's potential (so $tf$ is a summable Walters's potential, for each $t>0$). Then (following \cite{W} corollary 3.6 and theorem 1.1) there exists $h_{t}\in C(X)$, $h_{t}>0$, with $L_{tf}h_{t} = \lambda_{t} .h_{t}$, $\lambda_{t} = e^{P(tf)}>0$. Also, there exists a measure $\nu_{t}$ with $L_{tf}^{*}\nu_{t} = \lambda_{t}.\nu_{t}$. If we normalise $h_{t}$ by $h_{t}(0^{\infty})=1$, and, $\nu_{t}$ by $\int h_{t} \, d\nu_{t} =1$, then, the Gibbs measure $d\mu_{t} := h_{t}\,d\nu_{t}$ is an equilibrium state for $tf$.

Following Walters (page 1341 and corollary 3.5 \cite{W}) it is known that $P(tf)>\max\{ta,tc\}$ and
\scriptsize
\begin{equation}
 \left(\sum_{j=0}^{\infty}e^{td_{1+j}+t(a_{2}+...+a_{1+j}) - jP(tf)}\right)\left(\sum_{j=0}^{\infty}e^{tb_{1+j}+t(c_{2}+...+c_{1+j}) - jP(tf)}\right)=e^{2P(tf)}. \label{walters1}
 \end{equation}
 \normalsize

By abuse of notation, in the above expression, when $j=0$, the corresponding terms are just $e^{td_{1}}$ and $e^{tb_{1}}$, respectively.

\bigskip

 Moreover, the eigenfunction $h_{t}$ of a summable Walters's Potential $tf$ is given by \cite{W}:
\[h_{t}(0^{\infty}) = 1.\]
\[h_{t}(1^{\infty}) =: \beta_{\infty,t} := \frac{e^{tb}(e^{P(tf)}-e^{ta})}{e^{td}(e^{P(tf)} - e^{tc})e^{P(tf)}}\left(\sum_{j=0}^{\infty}e^{td_{1+j}+t(a_{2}+...+a_{1+j}) - jP(tf)}\right) \]
\[h_{t}(0^{q}1) =: \alpha_{q,t} :=  \frac{(e^{P(tf)}-e^{ta})}{e^{td}e^{P(tf)}}\left(\sum_{j=0}^{\infty}e^{td_{q+j}+t(a_{q+1}+...+a_{q+j}) - jP(tf)}\right)\]
\[h_{t}(1^{q}0) =: \beta_{q,t} :=  \frac{\beta_{\infty,t}(e^{P(tf)}-e^{tc})}{e^{tb}e^{P(tf)}}\left(\sum_{j=0}^{\infty}e^{tb_{q+j}+t(c_{q+1}+...+c_{q+j}) - jP(tf)}\right).\]

 Remember that $\beta(f) = \sup_{\mu \, \in {\cal M}}\int f \, d\mu$. In the present setting $\beta(f) = a=c$.
We denote $\epsilon_{t}:=P(tf) - t\beta(f).$ Under the assumption ``any maximizing measure is of the form  $s\,\delta_{0^{\infty}} + (1-s)\,\delta_{1^{\infty}}, \ \ s \in [0,1]$'' (note that $\epsilon_{t}>0$) we also have  $\epsilon_{t}\to 0$ (the supremum of the entropy of the maximizing invariant probabilities is zero \cite{CoG}).
On the analysis of the eventual existence of the limit $\frac{1}{t}\log(h_{t})$, it's important to understand first what happens with $\frac{1}{t}\log(\epsilon_{t})$.
\bigskip

\begin{lemma} \label{e}
Suppose $f$ is a summable Walters's potential, $a=c=\beta(f)$, and that there exists
\[\lim_{t_{i}\to\infty}\frac{1}{t_{i}}\log(\epsilon_{t_{i}}).\] Then, for fixed $q$ we get
\begin{align*}
&\frac{1}{t_{i}}\log  \left(\sum_{j=0}^{\infty}e^{t_{i}d_{q+j} +t_{i}(a_{q+1}+...+a_{q+j}) - jP(t_{i}f)}\right)\\
&= \max\{\sup_{j\geq 0}(d_{q+j} + a_{q+1}+...+a_{q+j} - j\beta(f)), d+ \sum_{j=1}^{\infty}(a_{q+j}-\beta(f)) -\lim_{t_{i}\to\infty}\frac{1}{t_{i}}\log\epsilon_{t_{i}}\}.
\end{align*}
and
\begin{align*}
&\frac{1}{t_{i}}\log  \left(\sum_{j=0}^{\infty}e^{t_{i}b_{q+j} +t_{i}(c_{q+1}+...+c_{q+j}) - jP(t_{i}f)}\right)\\
&= \max\{\sup_{j\geq 0}(b_{q+j} + c_{q+1}+...+c_{q+j} - j\beta(f)), b+ \sum_{j=1}^{\infty}(c_{q+j}-\beta(f)) -\lim_{t_{i}\to\infty}\frac{1}{t_{i}}\log\epsilon_{t_{i}}\}.
\end{align*}
\end{lemma}

\begin{proof}
By abuse of notation we write $t$ instead of $t_i$. We will only prove the first equality because the second case is analogous.
We fix $\epsilon >0$. By assumption, the series $\sum_{n\geq 2}(a_{n}-a)$ is convergent, so there exist $n$ such that for all $j\geq n$, $$|d_{j}-d|\,<\epsilon/2, \ \ and \ \ |a_{n}+...+a_{j} - (1+j-n)a| <\epsilon/2.$$
It follows that

\begin{align*}
\frac{1}{t}\log & \left(\sum_{j=0}^{\infty}e^{td_{q+j}}\left(\frac{e^{ta_{q+1}+...+ta_{q+j}}}{(\lambda(t))^{j}}\right)\right)\\
&= \frac{1}{t}\log  \left(\sum_{j=0}^{\infty}e^{td_{q+j}+t(a_{q+1}+...+a_{q+j})-tj\beta(f) - j\epsilon_{t}}\right)\\
&= \frac{1}{t}\log \left(\sum_{j=0}^{\infty}e^{td_{q+j}+t(a_{q+1}+...+a_{q+j})-tja - j\epsilon_{t}}\right)\\
&= \frac{1}{t}\log \left(\sum_{j=0}^{n}e^{td_{q+j}+t(a_{q+1}+...+a_{q+j})-tja - j\epsilon_{t}} + \sum_{j=n+1}^{\infty}e^{td_{q+j}+t(a_{q+1}+...a_{q+j})-tja - j\epsilon_{t}} \right).\\
\end{align*}

Moreover,
\begin{align*}
&\limsup_{t\to\infty}\frac{1}{t}\log \left(\sum_{j=0}^{n}e^{td_{q+j}+t(a_{q+1}+...+a_{q+j})-tja - j\epsilon_{t}} + \sum_{j=n+1}^{\infty}e^{td_{q+j}+t(a_{q+1}+...+a_{q+j})-tja - j\epsilon_{t}} \right)\\
&\leq \limsup_{t\to\infty}\frac{1}{t}\log \left(\sum_{j=0}^{n}e^{td_{q+j}+t(a_{q+1}+...+a_{q+j})-tja - j\epsilon_{t}} + e^{td+t(a_{q+1}+...+a_{q+n})-tn\beta(f)}\sum_{j=n+1}^{\infty}e^{t\epsilon - j\epsilon_{t}} \right)\\
&\leq \limsup_{t\to\infty}\frac{1}{t}\log \left(\sum_{j=0}^{n}e^{td_{q+j}+t(a_{q+1}+...+a_{q+j})-tja - j\epsilon_{t}} + e^{td+t(a_{q+1}+...+a_{q+n})-tn\beta(f)+t\epsilon}\frac{e^{-(n+1)\epsilon_{t}}}{1-e^{-\epsilon_{t}}} \right)\\
&= \max\{\sup_{0\leq j\leq n}(d_{q+j} + a_{q+1}+...+a_{q+j} - j\beta(f)), d + \epsilon +\sum_{j=q+1}^{q+n} (a_{j}-\beta(f)) -\lim_{t\to\infty}\frac{1}{t}\log\epsilon_{t}\}.
\end{align*}

This is true for any $\epsilon>0$, and $n$ sufficiently large, so
\begin{align*}
&\limsup_{t\to\infty}\frac{1}{t}\log  \left(\sum_{j=0}^{\infty}e^{td_{q+j}}\left(\frac{e^{ta_{q+1}+...+ta_{q+j}}}{(\lambda(t))^{j}}\right)\right)\\
&\leq \max\{\sup_{j}(d_{q+j} + a_{q+1}+...+a_{q+j} - j\beta(f)), d+ \sum_{j=1}^{\infty}(a_{q+j}-\beta(f)) -\lim_{t\to\infty}\frac{1}{t}\log\epsilon_{t}\}.
\end{align*}

We also have that
\begin{align*}
&\liminf_{t\to\infty}\frac{1}{t}\log  \left(\sum_{j=0}^{\infty}e^{td_{q+j}}\left(\frac{e^{ta_{q+1}+...+ta_{q+j}}}{(\lambda(t))^{j}}\right)\right)\\
&= \liminf_{t\to\infty}\frac{1}{t}\log \left(\sum_{j=0}^{n}e^{td_{q+j}+t(a_{q+1}+...+a_{q+j})-tja - j\epsilon_{t}} + \sum_{j=n+1}^{\infty}e^{td_{q+j}+t(a_{q+1}+...a_{q+j})-tja - j\epsilon_{t}} \right)\\
&\geq \liminf_{t\to\infty}\frac{1}{t}\log \left(\sum_{j=0}^{n}e^{td_{q+j}+t(a_{q+1}+...+a_{q+j})-tja - j\epsilon_{t}} + e^{td+t(a_{q+1}+...+a_{q+n})-tn\beta(f)}\sum_{j=n+1}^{\infty}e^{-t\epsilon - j\epsilon_{t}} \right)\\
&\geq \liminf_{t\to\infty}\frac{1}{t}\log \left(\sum_{j=0}^{n}e^{td_{q+j}+t(a_{q+1}+...+a_{q+j})-tja - j\epsilon_{t}} + e^{td+t(a_{q+1}+...+a_{q+n})-tn\beta(f)-t\epsilon}\frac{e^{-(n+1)\epsilon_{t}}}{1-e^{-\epsilon_{t}}} \right)\\
&= \max\{\sup_{0\leq j\leq n}(d_{q+j} + a_{q+1}+...+a_{q+j} - j\beta(f)), d - \epsilon+ \sum_{j=q+1}^{q+n} (a_{j}-\beta(f)) -\lim_{t\to\infty}\frac{1}{t}\log\epsilon_{t}\}.
\end{align*}

This is true for any $\epsilon>0$, and $n$ sufficiently large, so
\begin{align*}
&\lim_{t\to\infty}\frac{1}{t}\log \left(\sum_{j=0}^{\infty}e^{td_{q+j}}\left(\frac{e^{ta_{q+1}+...+ta_{q+j}}}{(\lambda(t))^{j}}\right)\right)\\
&= \max\{\sup_{j}(d_{q+j} + a_{q+1}+...+a_{q+j} - j\beta(f)), d+ \sum_{j=1}^{\infty}(a_{q+j}-\beta(f)) -\lim_{t\to\infty}\frac{1}{t}\log\epsilon_{t}\}.
\end{align*}
\end{proof}

\begin{lemma}\label{epsilont}
Suposse $f$ is a summable Walters's potential,  $a=c=\beta(f)$, and any maximizing measure is of the form  $s\,\delta_{0^{\infty}} + (1-s)\,\delta_{1^{\infty}}, \ \ s \in [0,1]$.  Then there exists
\[A:=\lim_{t\to\infty}\frac{1}{t}\log(\epsilon_{t}).\]
Also, $A$ is the unique number $\leq 0$ which satisfies the bellow equation on $y$:
\footnotesize
\begin{align}
2\beta(f) =& \max\{ \sup_{j\geq 0}(d_{1+j} + a_{2}+...+a_{1+j} - j\beta(f)), d+ \sum_{j=1}^{\infty}(a_{1+j}-\beta(f)) -y\}+ \nonumber \\
& + \max\{\sup_{j\geq 0}(b_{1+j} + c_{2}+...+c_{1+j} - j\beta(f)), b+ \sum_{j=1}^{\infty}(c_{1+j}-\beta(f)) -y\}. \label{eqA}
\end{align}
\normalsize
Moreover, $A=0$, or
\[A = A_{1}:= \frac{d+b}{2} + \frac{\sum_{j=1}^{\infty}(a_{1+j}-\beta(f))}{2} + \frac{\sum_{j=1}^{\infty}(c_{1+j}-\beta(f))}{2} -\beta(f),\]
or
\[A = A_{2}:= b + \sum_{j=1}^{\infty}(c_{1+j}-\beta(f)) + \sup_{j\geq 0}(d_{1+j} + a_{2}+...+a_{1+j} - j\beta(f)) - 2\beta(f),\]
or
\[A = A_{3}:= d + \sum_{j=1}^{\infty}(a_{1+j}-\beta(f)) + \sup_{j\geq 0}(b_{1+j} + c_{2}+...+c_{1+j} - j\beta(f)) - 2\beta(f).\]
\end{lemma}

\begin{proof}

Let $A$ be an accumulation point of $\lim_{t\to\infty}\frac{1}{t}\log(\epsilon_{t})$. Suppose the accumulation occurs for a certain sequence $t_{i}$.
We know by (pag 1342 \cite{W}) that
\scriptsize
\[ \left(\sum_{j=0}^{\infty}e^{td_{1+j}+t(a_{2}+...+a_{1+j}) - jP(tf)}\right)\left(\sum_{j=0}^{\infty}e^{tb_{1+j}+t(c_{2}+...+c_{1+j}) - jP(tf)}\right)=e^{2P(tf)}. \ \ \ \ (*)\]
\normalsize
So taking $\lim_{t_{i}\to\infty}\frac{1}{t_{i}}\log$ in both sides, and applying the above lemma, we have:
\begin{align*}
2\beta(f) =& \max\{  \sup_{j\geq 0}(d_{1+j} + a_{2}+...+a_{1+j} - j\beta(f)), d+ \sum_{j=1}^{\infty}(a_{1+j}-\beta(f)) -A\} \\
& + \max\{\sup_{j\geq 0}(b_{1+j} + c_{2}+...+c_{1+j} - j\beta(f)), b+ \sum_{j=1}^{\infty}(c_{1+j}-\beta(f)) -A\}.
\end{align*}

Then, any accumulation point of $\frac{1}{t}\log(\epsilon_{t})$ satisfies the equation $(\ref{eqA})$, and it is also small, or equal to  zero.

We are going to prove that there is a unique number small, or equal to zero, satisfying the equation $(\ref{eqA})$. \bigskip

\textbf{First case:}
\[2\beta(f) =( \sup_{j\geq 0}\{(d_{1+j} + a_{2}+...+a_{1+j} - j\beta(f))\})+( \sup_{j\geq 0}\{(b_{1+j} + c_{2}+...+c_{1+j} - j\beta(f))\}) .\]
We will show that, on this case, $A=0$ is the unique value smaller than zero which satisfies $(\ref{eqA})$.
If
\[\sup_{j\geq 0}d_{1+j} + a_{2}+...+a_{1+j} - j\beta(f) = d+ \sum_{j=1}^{\infty}(a_{1+j}-\beta(f)),\]
or
\[\sup_{j\geq 0}b_{1+j} + c_{2}+...+c_{1+j} - j\beta(f) = b+ \sum_{j=1}^{\infty}(c_{1+j}-\beta(f)),\]
then, clearly, the claim is true.
If this does not happen, then there exist $j_{0}$ and $j_{1}$ such that
\[\sup_{j\geq 0}(d_{1+j} + a_{2}+...+a_{1+j} - j\beta(f)) = d_{1+j_{0}} + a_{2}+...+a_{1+j_{0}} - j_{0}\beta(f),\]
and,
\[ \sup_{j\geq 0}(b_{1+j} + c_{2}+...+c_{1+j} - j\beta(f)) = b_{1+j_{1}} + c_{2}+...+c_{1+j_{1}} - j_{1}\beta(f).\]
So, we have
\[2\beta(f) = d_{1+j_{0}} + a_{2}+...+a_{1+j_{0}} - j_{0}\beta(f) + b_{1+j_{1}} + c_{2}+...+c_{1+j_{1}} - j_{1}\beta(f).\]
It follows that
\[\beta(f) = \frac{d_{1+j_{0}} + a_{2}+...+a_{1+j_{0}} + b_{1+j_{1}} + c_{2}+...+c_{1+j_{1}}}{j_{0}+j_{1}+2} = \frac{f^{j_{0}+j_{1}+2}((0^{j_{0}+1}1^{j_{1}+1})^{\infty})}{j_{0}+j_{1}+2}.\]

Therefore,  there exists a large periodic orbit which attains the maximal value.
This can not happen because any maximizing is of the form $s\,\delta_{0^{\infty}} + (1-s)\,\delta_{1^{\infty}}, \ \ s \in [0,1]$. We can not have other periodic orbits getting such supremum. Then, on this  case, the claim is true.
\bigskip

\textbf{Second case:}%
\[2\beta(f) > ( \sup_{j\geq 0}(d_{1+j} + a_{2}+...+a_{1+j} - j\beta(f)))+( \sup_{j\geq 0}(b_{1+j} + c_{2}+...+c_{1+j} - j\beta(f))) .\]
On this case, $y=0$  is not a solution of $(\ref{eqA})$. We are going to prove that, if $y_{1}$ is a solution of $(\ref{eqA})$, and, $y_{1}<y_{2}$, then $y_{2}$ is not a solution of $(\ref{eqA})$. Indeed,

\begin{align*}
(d+ & \sum_{j=1}^{\infty}(a_{1+j}-\beta(f)) -y_{2}) + ( b+ \sum_{j=1}^{\infty}(c_{1+j}-\beta(f)) -y_{2})\\
& <  (d+ \sum_{j=1}^{\infty}(a_{1+j}-\beta(f)) -y_{1}) + ( b+ \sum_{j=1}^{\infty}(c_{1+j}-\beta(f)) -y_{1})\\
&\leq \max\{  \sup_{j\geq 0}(d_{1+j} + a_{2}+...+a_{1+j} - j\beta(f)), d+ \sum_{j=1}^{\infty}(a_{1+j}-\beta(f)) -y_{1}\} \\
& + \max\{\sup_{j\geq 0}(b_{1+j} + c_{2}+...+c_{1+j} - j\beta(f)), b+ \sum_{j=1}^{\infty}(c_{1+j}-\beta(f)) -y_{1}\} = 2\beta(f),
\end{align*}
and,
\begin{align*}
(\sup_{j\geq 0}&(d_{1+j} + a_{2}+...+a_{1+j} - j\beta(f))) + ( b+ \sum_{j=1}^{\infty}(c_{1+j}-\beta(f)) -y_{2})\\
&< (\sup_{j\geq 0}(d_{1+j} + a_{2}+...+a_{1+j} - j\beta(f))) + ( b+ \sum_{j=1}^{\infty}(c_{1+j}-\beta(f)) -y_{1})\\
&\leq \max\{ \sup_{j\geq 0}(d_{1+j} + a_{2}+...+a_{1+j} - j\beta(f)), d+ \sum_{j=1}^{\infty}(a_{1+j}-\beta(f)) -y_{1}\} \\
& + \max\{\sup_{j\geq 0}(b_{1+j} + c_{2}+...+c_{1+j} - j\beta(f)), b+ \sum_{j=1}^{\infty}(c_{1+j}-\beta(f)) -y_{1}\} = 2\beta(f),
\end{align*}
and,
\begin{align*}
(d+ \sum_{j=1}^{\infty}&(a_{1+j}-\beta(f)) -y_{2}) + ( \sup_{j\geq 0}(b_{1+j} + c_{2}+...+c_{1+j} - j\beta(f)))\\
&< (d+ \sum_{j=1}^{\infty}(a_{1+j}-\beta(f)) -y_{1}) + ( \sup_{j\geq 0}(b_{1+j} + c_{2}+...+c_{1+j} - j\beta(f)))\\
&\leq \max\{ \sup_{j\geq 0}(d_{1+j} + a_{2}+...+a_{1+j} - j\beta(f)), d+ \sum_{j=1}^{\infty}(a_{1+j}-\beta(f)) -y_{1}\} \\
& + \max\{\sup_{j\geq 0}(b_{1+j} + c_{2}+...+c_{1+j} - j\beta(f)), b+ \sum_{j=1}^{\infty}(c_{1+j}-\beta(f)) -y_{1}\} = 2\beta(f).
\end{align*}
These three inequalities, plus the hypothesis of our second case, shows that $y_{2}$ is not a solution of $(\ref{eqA})$.
\bigskip

We can make a more precise analysis, and conclude that $A=0$, or,
alternatively, $A$ will take one of the three values given on the claim of the lemma:\newline
\newline

\textbf{First case:}
\[2\beta(f) = (d+ \sum_{j=1}^{\infty}(a_{1+j}-\beta(f)) -A) + ( b+ \sum_{j=1}^{\infty}(c_{1+j}-\beta(f)) -A),\]
and, so
\[A = \frac{d+b}{2} + \frac{\sum_{j=1}^{\infty}(a_{1+j}-\beta(f))}{2} + \frac{\sum_{j=1}^{\infty}(c_{1+j}-\beta(f))}{2} -\beta(f).\]
\newline

\textbf{Second case:}
\[2\beta(f) = (\sup_{j\geq 0}(d_{1+j} + a_{2}+...+a_{1+j} - j\beta(f))) + ( b+ \sum_{j=1}^{\infty}(c_{1+j}-\beta(f)) -A),\]
and, so
\[A = b + \sum_{j=1}^{\infty}(c_{1+j}-\beta(f)) + \sup_{j\geq 0}(d_{1+j} + a_{2}+...+a_{1+j} - j\beta(f)) - 2\beta(f).\]


\textbf{Third case:}
\[2\beta(f) = (d+ \sum_{j=1}^{\infty}(a_{1+j}-\beta(f)) -A) + ( \sup_{j\geq 0}(b_{1+j} + c_{2}+...+c_{1+j} - j\beta(f))),\]
and, so
\[A = d + \sum_{j=1}^{\infty}(a_{1+j}-\beta(f)) + \sup_{j\geq 0}(b_{1+j} + c_{2}+...+c_{1+j} - j\beta(f)) - 2\beta(f).\]

\end{proof}

\subsection{Proof of Theorem $\ref{teo1}$:}

The first part is contained in Lemma $\ref{epsilont}$. We have $V(0^{\infty})=0$,  because $h_{t}(0^{\infty})=1$.
We also have
\begin{align*}
V(&1^{\infty}) = \lim_{t\to\infty}\frac{1}{t}\log(\beta_{\infty,t}) \\
&=  \lim_{t\to\infty}\frac{1}{t} \log \left(\frac{e^{tb}(e^{P(tf)}-e^{ta})}{e^{td}(e^{P(tf)} - e^{tc})e^{P(tf)}}\right)\\
& \, + \lim_{t\to\infty}\frac{1}{t} \log \left(\sum_{j=0}^{\infty}e^{td_{1+j}+t(a_{2}+...+a_{1+j}) - jP(tf)}\right)\\
&= b - d  - \beta(f) + \max\{\sup_{j}(d_{1+j} + a_{2}+...+a_{1+j} - j\beta(f)), d+ \sum_{j=1}^{\infty}(a_{1+j}-\beta(f)) -A\}.
\end{align*}

Moreover,

\begin{align*}
V(& 0^{q}1z) = \lim_{t\to\infty}\frac{1}{t}\log(\alpha_{q,t}) \\
&=  \lim_{t\to\infty}\frac{1}{t} \log \left(\frac{(e^{P(tf)}-e^{ta})}{e^{td}e^{P(tf)}}\left(\sum_{j=0}^{\infty}e^{td_{q+j}+t(a_{q+1}+...+a_{q+j}) - jP(tf)}\right)\right)\\
&=  \lim_{t\to\infty}\frac{1}{t} \log \left(\frac{(e^{P(tf)}-e^{ta})}{e^{td}e^{P(tf)}}\right)\\
&\, + \lim_{t\to\infty}\frac{1}{t} \log \left(\sum_{j=0}^{\infty}e^{td_{q+j}+t(a_{q+1}+...+a_{q+j}) - jP(tf)}\right)\\
&= \lim_{t\to\infty}\frac{1}{t} \log (1-e^{-\epsilon_{t}}) - d \\
&\, + \max\{\sup_{j}(d_{q+j} + a_{q+1}+...+a_{q+j} - j\beta(f)), d+ \sum_{j=1}^{\infty}(a_{q+j}-\beta(f)) -A\}\\
&= A - d +  \max\{\sup_{j}(d_{q+j} + a_{q+1}+...+a_{q+j} - j\beta(f)), d+ \sum_{j=1}^{\infty}(a_{q+j}-\beta(f)) -A\},
\end{align*}
and, finally
\begin{align*}
 V(& 1^{q}0z) = \lim_{t\to\infty}\frac{1}{t} \log( \beta_{q,t})\\
&= \lim_{t\to\infty}\frac{1}{t} \log\left(\frac{\beta_{\infty}(t)(e^{P(tf)}-e^{tc})}{e^{tb}e^{P(tf)}}\right)\\
& \, + \lim_{t\to\infty}\frac{1}{t} \log \left(\sum_{j=0}^{\infty}e^{tb_{q+j}+t(c_{q+1}+...+c_{q+j}) - jP(tf)}\right)\\
&= -d -\beta(f) +A \\
& \ \ + \max\{\sup_{j}(d_{1+j} + a_{2}+...+a_{1+j} - j\beta(f)), d+ \sum_{j=1}^{\infty}(a_{1+j}-\beta(f)) -A\}\\
& \ \ + \max\{\sup_{j}(b_{q+j} + c_{q+1}+...+c_{q+j} - j\beta(f)), b+ \sum_{j=1}^{\infty}(c_{q+j}-\beta(f)) -A\}.
\end{align*}

We remark that it's possible verify explicitly that this function $V$ is a calibrated subaction for $f$.

\section{Analysis of the Gibbs state and  selection of an invariant measure}

First, we are going to show:

\begin{proposition} \label{quo}
Let $f$ be a Walters's summable Potential, and, $\mu_{t}$ the Gibbs measure of $tf$. Then
\[\mu_{t}([0]) = \frac{S_{0}(t)}{S_{0}(t)+S_{1}(t)}, \ \ and, \ \  \mu_{t}([1]) = \frac{S_{1}(t)}{S_{0}(t)+S_{1}(t)},\]
where
\[S_{0}(t):= \frac{ e^{td_{1}} + \sum_{j=1}^{\infty}(j+1)e^{td_{1+j} +t(a_{2}+...+a_{1+j})-jP(tf)}}{e^{td_{1}}+\sum_{j=1}^{\infty}e^{td_{1+j} +t(a_{2}+...+a_{1+j})-jP(tf)}},\]
\[S_{1}(t) = \frac{ e^{tb_{1}}+\sum_{j=1}^{\infty}(j+1)e^{tb_{1+j}+t(c_{2}+...+c_{1+j})-jP(tf)}}{ e^{tb_{1}}+\sum_{j=1}^{\infty}e^{tb_{1+j}+t(c_{2}+...+c_{1+j})-jP(tf)}}. \]
\end{proposition}

\textbf{Remark:} We know that $e^{td_{1}}+\sum_{j=1}^{\infty}e^{td_{1+j} +t(a_{2}+...+a_{1+j})-jP(tf)}<\infty$. We are going to show bellow that $\mu_{t}([0])=\mu_{t}([01])S_{0}$. So we have $S_{0}<\infty$. The same argument can be used for $S_{1}$.

First, we need some lemmas:

\begin{lemma} Under the above hypothesis
\[S_{0}(t)= 1+ \sum_{j=2}^{\infty} e^{-(j-1)P(tf) + t(a_{2}+...+a_{j}) + \log(\alpha_{j,t}) - \log(\alpha_{1,t})},\]
\[S_{1}(t)= 1+ \sum_{j=2}^{\infty} e^{-(j-1)P(tf) + t(c_{2}+...+c_{j}) + \log(\beta_{j,t}) - \log(\beta_{1,t})}.\]
\end{lemma}

\begin{proof}
We are going to show the equality for $S_{0}$ (the case for $S_{1}$ is similar). We have
\begin{align*}
1+ & \sum_{j=2}^{\infty} e^{-(j-1)P(tf) + t(a_{2}+...+a_{j}) + \log(\alpha_{j,t}) - \log(\alpha_{1,t})}\\
&= 1+ \frac{\sum_{j=2}^{\infty} e^{-(j-1)P(tf) + t(a_{2}+...+a_{j}) + \log(e^{td_{j}}+\sum_{i=1}^{\infty}e^{td_{j+i} +t(a_{j+1}+...+a_{j+i}) -iP(tf)})}}{ e^{td_{1}}+\sum_{i=1}^{\infty}e^{td_{1+i} +t(a_{2}+...+a_{1+i}) -iP(tf)}}\\
&= 1+ \frac{\sum_{j=2}^{\infty} \left(e^{-(j-1)P(tf) + t(a_{2}+...+a_{j})}\left( e^{td_{j}}+\sum_{i=1}^{\infty}e^{td_{j+i} +t(a_{j+1}+...+a_{j+i}) -iP(tf)}\right)\right)}{ e^{td_{1}}+\sum_{i=1}^{\infty}e^{td_{1+i} +t(a_{2}+...+a_{1+i}) -iP(tf)}}\\
&=1+ \frac{\sum_{j=2}^{\infty} \left( e^{td_{j}+t(a_{2}+...+a_{j})-(j-1)P(tf)}+\sum_{i=1}^{\infty}e^{td_{j+i} +t(a_{2}+...+a_{j+i}) -(j+i-1)P(tf)}\right)}{ e^{td_{1}}+\sum_{i=1}^{\infty}e^{td_{1+i} +t(a_{2}+...+a_{1+i}) -iP(tf)}}\\
&=1+ \frac{\sum_{j=2}^{\infty} \sum_{i=0}^{\infty}e^{td_{j+i} +t(a_{2}+...+a_{j+i}) -(j+i-1)P(tf)}}{ e^{td_{1}}+\sum_{i=1}^{\infty}e^{td_{1+i} +t(a_{2}+...+a_{1+i}) -iP(tf)}}\\
&=1+ \frac{\sum_{j=2}^{\infty}(j-1)e^{td_{j} +t(a_{2}+...+a_{j}) -(j-1)P(tf)}}{ e^{td_{1}}+\sum_{i=1}^{\infty}e^{td_{1+i} +t(a_{2}+...+a_{1+i}) -iP(tf)}}\\
&= \frac{e^{td_{1}}+\sum_{i=1}^{\infty}\left(e^{td_{1+i} +t(a_{2}+...+a_{1+i}) -iP(tf)}\right)+\sum_{j=2}^{\infty}(j-1)e^{td_{j} +t(a_{2}+...+a_{j}) -(j-1)P(tf)}}{ e^{td_{1}}+\sum_{i=1}^{\infty}e^{td_{1+i} +t(a_{2}+...+a_{1+i}) -iP(tf)}}\\
&=\frac{ e^{td_{1}} + \sum_{j=1}^{\infty}(j+1)e^{td_{1+j} +t(a_{2}+...+a_{1+j})-jP(tf)}}{e^{td_{1}}+\sum_{j=1}^{\infty}e^{td_{1+j} +t(a_{2}+...+a_{1+j})-jP(tf)}}.
\end{align*}

This shows the claim.

\end{proof}

\begin{lemma}
For any cylinder $K$
\begin{equation}
\mu_{t}(\sigma(K)) = \int_{K} e^{-tf - \log(h_{t}) + \log(h_{t}\circ\sigma) + P(tf)} \, d\mu_{t}.  \label{rel}
\end{equation}
\end{lemma}

\begin{proof}
See page 37 in \cite{PP}.

\end{proof}

Now, we can prove the claim of the above proposition:

\begin{proof}
Applying ($\ref{rel}$), for $n\geq 2$:
\begin{align*}
\mu_{t}([0^{n-1}1]) &= \mu_{t}(\sigma[0^{n}1]) = \int_{[0^{n}1]} e^{-tf - \log(h_{t}) + \log(h_{t}\circ\sigma) + P(tf)} \, d\mu_{t} \\
&= \mu_{t}([0^{n}1])e^{-ta_{n}-\log(\alpha_{n,t})+\log(\alpha_{n-1,t}) +P(tf)}.
\end{align*}

So,
\[\mu_{t}([0^{n}1]) = \mu_{t}([01])e^{-(n-1)P(tf)+t(a_{2}+...+a_{n}) +\log(\alpha_{n,t}) - \log(\alpha_{1,t})}.\]

Also, by the same argument
\[\mu_{t}([1^{n}0]) = \mu_{t}([10])e^{-(n-1)P(tf)+t(c_{2}+...+c_{n}) +\log(\beta_{n,t}) - \log(\beta_{1,t})}.\]

Moreover, as $\mu_{t}$ is non atomic, then
\[\mu_{t}([0]) = \sum_{j=1}^{\infty}\mu_{t}([0^{j}1]) = \mu_{t}([01])S_{0}(t).\]
\[\mu_{t}([1]) = \sum_{j=1}^{\infty}\mu_{t}([1^{j}0]) = \mu_{t}([10])S_{1}(t).\]
Using the fact that $\mu_{t}$ is $\sigma-$invariant:
\[\mu_{t}([01]) = \mu_{t}([0])-\mu_{t}([00]) = \mu_{t}([*0]) -\mu_{t}([00]) = \mu_{t}([10]).\]
So, finally
\[ 1 = \mu_{t}([0]) + \mu_{t}([1]) =  \mu_{t}([01])S_{0} + \mu_{t}([10])S_{1} = \mu_{t}([01])(S_{0}(t)+S_{1}(t)).\]
Then,
\[\mu_{t}([0]) = \frac{S_{0}(t)}{S_{0}(t)+S_{1}(t)}, \ \ and \ \  \mu_{t}([1]) = \frac{S_{1}(t)}{S_{0}(t)+S_{1}(t)}.\]
\end{proof}

\begin{remark} We will be interested, of course, in estimating the limit $S_{0}(t)/S_{1}(t)$.
\end{remark}

Now we will consider more general cylinder sets of size $3+n, \, n\geq 0$. The procedure will by induction: we suppose we are able to estimate $\mu_{t}(K)$ for cylinders  $K$ of size smaller than $ n+2$. Then, we have:

\begin{proposition}\label{propa}
For $j\geq 2$, $n\geq 0$
\[\mu_{t}([0^{j}1x_{1}...x_{n+2-j}]) = \mu_{t}([0^{j-1}1x_{1}...x_{n+2-j}])
e^{-P(tf)+ta_{j}+\log(\alpha_{j,t})-
\log(\alpha_{j-1,t})},\]
\begin{align*}
\mu_{t}([0^{n+3}]) &= \mu_{t}([0^{n+2}]) - \mu_{t}([0^{n+2}1])\\
&=\mu_{t}([0^{n+2}]) -\mu_{t}([0^{n+1}1])e^{-P(tf)+ta_{n+2}+\log(\alpha_{n+2,t})-\log(\alpha_{n+1,t})}.
\end{align*}
In the same way
\[\mu_{t}([1^{j}0x_{1}...x_{n+2-j}]) = \mu_{t}([1^{j-1}0x_{1}...x_{n+2-j}])
e^{-P(tf)+tc_{j}+\log(\beta_{j,t})-
\log(\beta_{j-1,t})},\]
\begin{align*}
\mu_{t}([1^{n+3}]) &= \mu_{t}([1^{n+2}]) - \mu_{t}([1^{n+2}0])\\
&= \mu_{t}([1^{n+2}]) - \mu_{t}([1^{n+1}0])e^{-P(tf)+tc_{n+2}+\log(\beta_{n+2,t})-\log(\beta_{n+1,t})}.
\end{align*}
\end{proposition}

\begin{proof}
 The first equality is a consequence of $(\ref{rel})$ and the second is obvious. The others expressions follow in a similar way.
\end{proof}

\begin{proposition}\label{propb}
For $j\geq 1$:
\begin{align*}
\mu_{t}([01^{j}0x_{1}...x_{n-j+1}]) &= \mu_{t}([1^{j}0x_{1}...x_{n-j+1}]) - \mu_{t}([1^{j+1}0x_{1}...x_{n-j+1}])\\
&= \mu_{t}([1^{j}0x_{1}...x_{n-j+1}])(1-e^{-P(tf) +tc_{j+1} + \log(\beta_{j+1,t})-\log(\beta_{j,t})      }     ).
\end{align*}
In the same way,
for $j\geq 1$:
\begin{align*}
\mu_{t}([10^{j}1x_{1}...x_{n-j+1}]) &= \mu_{t}([0^{j}1x_{1}...x_{n-j+1}]) - \mu_{t}([0^{j+1}1x_{1}...x_{n-j+1}]) \\
&= \mu_{t}([0^{j}1x_{1}...x_{n-j+1}])(1-e^{-P(tf) +ta_{j+1} + \log(\alpha_{j+1,t})-\log(\alpha_{j,t})      }     ).
\end{align*}

\end{proposition}
\begin{proof} It follows in a similar way as last proposition.
\end{proof}

\subsection{Analysis of the limit  $\lim_{t\to\infty} \frac{1}{t}\log(P(tf))$}

On this section we are going to consider a special class that will be called \textbf{Non-positive Potentials}. These  are the summable Walters's potentials which satisfies: $a_{n}<0$, $c_{n}<0$, $b_{n}\equiv b <0$ , $d_{n}\equiv d < 0$, and $a=c=0$.
\bigskip

We remember that $\epsilon_{t} := P(tf) -t\beta(f)$. For Non-positive Potentials  we have $P(tf)=\epsilon_{t} \to 0$ ($0$ is the supreme of the entropy of maximizing measures \cite{CoG}).

\begin{proposition}
For Non-positive Potentials there exists
\[\lim_{t\to\infty}\frac{1}{t}\log(\epsilon_{t}).\]
Moreover, it converges to
\[\left\{ \begin{array}{ll}
b+d +\sum_{j=1}^{\infty}c_{1+j}, & \text{when} \ \ \sum_{j=1}^{\infty}a_{1+j}\leq b+d +\sum_{j=1}^{\infty}c_{1+j},\\
b+d +\sum_{j=1}^{\infty}a_{1+j}, & \text{when} \ \ \sum_{j=1}^{\infty}c_{1+j} \leq b+d +\sum_{j=1}^{\infty}a_{1+j},\\
\frac{b+d}{2} +\sum_{j=1}^{\infty}\frac{a_{1+j}}{2} +\sum_{j=1}^{\infty}\frac{c_{1+j}}{2}, & \text{in the other case}
\end{array} \right.\]
 We denote by $A$ any the corresponding limit
$\lim_{t\to\infty}\frac{1}{t}\log(\epsilon_{t}).$

\end{proposition}

\textbf{Remark:}
Note that the cases $\sum_{j=1}^{\infty}a_{1+j} \leq b+d +\sum_{j=1}^{\infty}c_{1+j}$ and \newline
$\sum_{j=1}^{\infty}c_{1+j} \leq b+d +\sum_{j=1}^{\infty}a_{1+j}$ can not occur simultaneously, because $2(b+d) < 0$.

\begin{proof}
We apply the claim of Lemma $\ref{epsilont}$ and conclude that for Non-positive Potentials the limit $A:=\lim_{t\to\infty}\frac{1}{t}\log(\epsilon_{t})$ is the unique real number smaller than zero satisfying
\[\max\{d, d+\sum_{j=1}^{\infty}a_{1+j} - A\} + \max\{b, b+\sum_{j=1}^{\infty}c_{1+j} - A\} =  0. \ \ (*)\]

 Moreover, by Lemma $\ref{epsilont}$ (note that $A\neq 0$, as one can see from analysis of the first case of lemma \ref{epsilont}) there are three possibilities:
\[ A=A_{1}= \frac{b+d}{2} +\sum_{j=1}^{\infty}\frac{a_{1+j}}{2} +\sum_{j=1}^{\infty}\frac{c_{1+j}}{2}\]
or,
\[A = A_{2} = b+d +\sum_{j=1}^{\infty}c_{1+j}, \]
 or,
\[ A =A_{3}= b+d +\sum_{j=1}^{\infty}a_{1+j}.\]

\textbf{First case:}
\[\sum_{j=1}^{\infty}a_{1+j} \leq b+d +\sum_{j=1}^{\infty}c_{1+j}.\]
Then,
\[\sum_{j=1}^{\infty}c_{1+j} > b+d +\sum_{j=1}^{\infty}a_{1+j}.\]
We first show that, in this case, $A\neq A_{3}$, because $A_{3}$ don't satisfies (*). Indeed,
\begin{align*}
&\max\{d, d+\sum_{j=1}^{\infty}a_{1+j} - A_{3}\} + \max\{b, b+\sum_{j=1}^{\infty}c_{1+j} - A_{3}\}\\
& =\max\{d, d+\sum_{j=1}^{\infty}a_{1+j} -b -d -\sum_{j=1}^{\infty}a_{1+j} \} + \max\{b, b+\sum_{j=1}^{\infty}c_{1+j} - b-d -\sum_{j=1}^{\infty}a_{1+j}\}\\
& = (-b )+ ( b+\sum_{j=1}^{\infty}c_{1+j} - b-d -\sum_{j=1}^{\infty}a_{1+j})  \\
&= \sum_{j=1}^{\infty}c_{1+j} - b-d -\sum_{j=1}^{\infty}a_{1+j} >0.
\end{align*}
Now, we analyze $A_{1}$. Note that
\begin{align*}
&\max\{d, d+\sum_{j=1}^{\infty}a_{1+j} - A_{1}\} + \max\{b, b+\sum_{j=1}^{\infty}c_{1+j} - A_{1}\}\\
& =\max\{d, d+\sum_{j=1}^{\infty}a_{1+j} -\frac{b+d}{2} -\sum_{j=1}^{\infty}\frac{a_{1+j}}{2} -\sum_{j=1}^{\infty}\frac{c_{1+j}}{2} \}\\
& \, + \max\{b, b+\sum_{j=1}^{\infty}c_{1+j} - \frac{b+d}{2} -\sum_{j=1}^{\infty}\frac{a_{1+j}}{2} -\sum_{j=1}^{\infty}\frac{c_{1+j}}{2}\}\\
&= (d) + (b+\sum_{j=1}^{\infty}c_{1+j} - \frac{b+d}{2} -\sum_{j=1}^{\infty}\frac{a_{1+j}}{2} -\sum_{j=1}^{\infty}\frac{c_{1+j}}{2})\\
&= \frac{b+d}{2} + \sum_{j=1}^{\infty}\frac{c_{1+j}}{2} - \sum_{j=1}^{\infty}\frac{a_{1+j}}{2} \geq 0.
\end{align*}
If the equality is false, we get that $A=A_{2}$. If the equality is true, then
 \[A_{1}= \frac{b+d}{2} +\sum_{j=1}^{\infty}\frac{a_{1+j}}{2} +\sum_{j=1}^{\infty}\frac{c_{1+j}}{2} = b+d+\sum_{j=1}^{\infty}c_{1+j} = A_{2}.\]

\textbf{Second case:}
\[\sum_{j=1}^{\infty}c_{1+j} \leq b+d +\sum_{j=1}^{\infty}a_{1+j}.\]

The claim of this case follows from a similar argument as above.
\newline

\textbf{Third case:}
\[\sum_{j=1}^{\infty}c_{1+j} > b+d +\sum_{j=1}^{\infty}a_{1+j} \ \ and \ \ \sum_{j=1}^{\infty}a_{1+j} > b+d +\sum_{j=1}^{\infty}c_{1+j}.\]
We are going to show that $A_{2}$ does not satisfies (*). The proof for $A_{3}$ is analogous.

Note that
\begin{align*}
&\max\{d, d+\sum_{j=1}^{\infty}a_{1+j} - A_{2}\} + \max\{b, b+\sum_{j=1}^{\infty}c_{1+j} - A_{2}\}\\
&= \max\{d, d+\sum_{j=1}^{\infty}a_{1+j} - b-d -\sum_{j=1}^{\infty}c_{1+j}\} + \max\{b, b+\sum_{j=1}^{\infty}c_{1+j} - b-d -\sum_{j=1}^{\infty}c_{1+j}\}\\
&= (d+\sum_{j=1}^{\infty}a_{1+j} - b-d -\sum_{j=1}^{\infty}c_{1+j}) + (-d)\\
&= \sum_{j=1}^{\infty}a_{1+j} - b-d -\sum_{j=1}^{\infty}c_{1+j} > 0.
\end{align*}

The claim follows from this.
\end{proof}

\subsection{ The proof of Theorem $\ref{teo2}$}

Now, we analyze the selection of a measure by the family $\mu_t$, when $t\to\infty$, for a class of Non-positive potentials. We know that for any Walters potential:
\[\mu_{t}([0]) = \frac{S_{0}(t)}{S_{0}(t)+S_{1}(t)} \ \ \ \ and \ \ \ \ \mu_{t}([1]) = \frac{S_{1}(t)}{S_{0}(t)+S_{1}(t)},\]
where
\[S_{0}(t) = \frac{ e^{td_{1}} + \sum_{j=1}^{\infty}(j+1)e^{td_{1+j} +t(a_{2}+...+a_{1+j})-jP(tf)}}{e^{td_{1}}+\sum_{j=1}^{\infty}e^{td_{1+j} +t(a_{2}+...+a_{1+j})-jP(tf)}},\]
\[S_{1}(t) = \frac{ e^{tb_{1}}+\sum_{j=1}^{\infty}(j+1)e^{tb_{1+j}+t(c_{2}+...+c_{1+j})-jP(tf)}}{ e^{tb_{1}}+\sum_{j=1}^{\infty}e^{tb_{1+j}+t(c_{2}+...+c_{1+j})-jP(tf)}}. \]

For Non-positive Potentials we have
\[S_{0}(t) = \frac{ e^{td} + \sum_{j=1}^{\infty}(j+1)e^{td +t(a_{2}+...+a_{1+j})-j\epsilon_{t}}}{e^{td}+\sum_{j=1}^{\infty}e^{td +t(a_{2}+...+a_{1+j})-j\epsilon_{t}}},\]
\[=\frac{ 1 + \sum_{j=1}^{\infty}(j+1)e^{t(a_{2}+...+a_{1+j})-j\epsilon_{t}}}{1+\sum_{j=1}^{\infty}e^{t(a_{2}+...+a_{1+j})-j\epsilon_{t}}},\]
\[S_{1}(t) = \frac{ 1+\sum_{j=1}^{\infty}(j+1)e^{t(c_{2}+...+c_{1+j})-j\epsilon_{t}}}{ 1+\sum_{j=1}^{\infty}e^{t(c_{2}+...+c_{1+j})-j\epsilon_{t}}}. \]

We have to analyze what happens with the limit $\frac{S_{0}(t)}{S_{1}(t)}$, $t \to \infty$,  in order to prove  Theorem $\ref{teo2}$ .

\begin{lemma}
For $z<0$, we have, for any $j_1>0$
\[\sum_{j=0}^{\infty} (j+1) e^{jz} = \frac{1}{(1-e^{z})^{2}}.\]
\[ \sum_{j=j_{1}}^{\infty} (j+1) e^{jz} = e^{j_{1}z}\left(\frac{j_{1}}{1-e^{z}}+\frac{1}{(1-e^{z})^{2}}\right) .\]
\end{lemma}

\begin{proof} Note that
\[\sum_{j=0}^{\infty} (j+1) e^{jz} = \sum_{i=0}^{\infty}\sum_{j=i}^{\infty}e^{jz} = \sum_{i=0}^{\infty}\frac{e^{iz}}{1-e^{z}}= \frac{1}{(1-e^{z})^{2}}.\]

Moreover,

\begin{align*}
\sum_{j=j_{1}}^{\infty} (j+1) e^{jz} &= \sum_{i=0}^{\infty} (i+j_{1}+1) e^{(i+j_{1})z} = j_{1}e^{j_{1}z}\sum_{i=0}^{\infty}e^{iz} + e^{j_{1}(z)}\sum_{i=0}^{\infty}(i+1)e^{iz} \\
&= e^{j_{1}z}\left(\frac{j_{1}}{1-e^{z}}+\frac{1}{(1-e^{z})^{2}}\right).
\end{align*}

\end{proof}

\begin{corollary}\label{cora} For any $j_1>0$
\[\lim_{t\to\infty}(\epsilon_{t})^2\sum_{j=j_{1}}^{\infty} (j+1) e^{-j\epsilon_{t}} = 1.\]
\[\lim_{t\to\infty}\epsilon_{t}
\sum_{j=j_{1}}^{\infty} e^{-j\epsilon_{t}} = 1.\]
\end{corollary}

\begin{proof}
It's a consequence of the above Lemma, and the fact that it is true that $\epsilon_{t}(1-e^{-\epsilon_{t}})^{-1}\to 1,$ as $t \to \infty$.
\end{proof}

\textbf{Proof of Theorem $\ref{teo2}$:}

We suppose that $$\sum_{j=1}^{\infty}a_{1+j} < b + d + \sum_{j=1}^{\infty}c_{1+j},$$ and, therefore, $$A:=\lim_{t\to\infty}\frac{1}{t}\log(\epsilon_{t}) = b + d + \sum_{j=1}^{\infty}c_{1+j}.$$

So, we can write $\epsilon_{t} = e^{t\,(\,d+b+\sum_{j=1}^{\infty}c_{1+j}\,)\,  + \psi(t)}$, where $\psi(t)/t \to 0$.
We are going to show that
\[\limsup_{t\to\infty} \frac{\frac{ 1 + \sum_{j=1}^{\infty}(j+1)e^{t(a_{2}+...+a_{1+j})-j\epsilon_{t}}}{1+\sum_{j=1}^{\infty}e^{t(a_{2}+...+a_{1+j})-j\epsilon_{t}}}}  {\frac{ 1+\sum_{j=1}^{\infty}(j+1)e^{t(c_{2}+...+c_{1+j})-j\epsilon_{t}}}{ 1+\sum_{j=1}^{\infty}e^{t(c_{2}+...+c_{1+j})-j\epsilon_{t}}}} = 0.\]

Given $\epsilon>0$, there exists $n_{0}$, such that for all $n\geq n_{0}$,
\[\sum_{j=1}^{\infty} a_{1+j} < \sum_{j=1}^{n} a_{1+j} < \sum_{j=1}^{\infty} a_{1+j} +\epsilon \ \ and \ \  \sum_{j=1}^{\infty} c_{1+j}<\sum_{j=1}^{n} c_{1+j} < \sum_{j=1}^{\infty} c_{1+j} +\epsilon.\]
So, we can write
\[\limsup_{t\to\infty} \frac{\frac{ 1 + \sum_{j=1}^{\infty}(j+1)e^{t(a_{2}+...+a_{1+j})-j\epsilon_{t}}}{1+\sum_{j=1}^{\infty}e^{t(a_{2}+...+a_{1+j})-j\epsilon_{t}}}}  {\frac{ 1+\sum_{j=1}^{\infty}(j+1)e^{t(c_{2}+...+c_{1+j})-j\epsilon_{t}}}{ 1+\sum_{j=1}^{\infty}e^{t(c_{2}+...+c_{1+j})-j\epsilon_{t}}}} \]
\[\leq \limsup_{t\to\infty}(1 + \sum_{j=1}^{\infty}(j+1)e^{t(a_{2}+...+a_{1+j})-j\epsilon_{t}})\frac{1+\sum_{j=1}^{\infty}e^{t(c_{2}+...+c_{1+j})-j\epsilon_{t}}}{
1+\sum_{j=1}^{\infty}(j+1)e^{t(c_{2}+...+c_{1+j})-j\epsilon_{t}}}\]
\scriptsize
\begin{equation}
\leq \left(\limsup_{t\to\infty}(1 + \sum_{j=1}^{\infty}(j+1)e^{t(a_{2}+...+a_{1+j})-j\epsilon_{t}})\right)\left(\limsup_{t\to\infty}\frac{1+\sum_{j=1}^{\infty}e^{t(c_{2}+...+c_{1+j})-j\epsilon_{t}}}{
1+\sum_{j=1}^{\infty}(j+1)e^{t(c_{2}+...+c_{1+j})-j\epsilon_{t}}}\right) \label{eq3}
\end{equation}
\normalsize
First, we analise the second term:
\[\limsup_{t\to\infty}\frac{1+\sum_{j=1}^{\infty}e^{t(c_{2}+...+c_{1+j})-j\epsilon_{t}}}{
1+\sum_{j=1}^{\infty}(j+1)e^{t(c_{2}+...+c_{1+j})-j\epsilon_{t}}}\]
\[ \leq \limsup_{t\to\infty} \frac{1+\left(\sum_{j=1}^{n_{0}-1}e^{t(c_{2}+...+c_{1+j})-j\epsilon_{t}}\right) + e^{t\sum_{j=1}^{\infty} c_{1+j} +t\epsilon}\sum_{j=n_{0}}^{\infty}e^{-j\epsilon_{t}}}{
1+\sum_{j=1}^{n_{0}-1}(j+1)e^{t(c_{2}+...+c_{1+j})-j\epsilon_{t}} + e^{t\sum_{j=1}^{\infty} c_{1+j} }\sum_{j=n_{0}}^{\infty}(j+1)e^{-j\epsilon_{t}}}\]
\[=\limsup_{t\to\infty} \frac{1+ e^{t\sum_{j=1}^{\infty} c_{1+j} +t\epsilon}\sum_{j=n_{0}}^{\infty}e^{-j\epsilon_{t}}}{
1+e^{t\sum_{j=1}^{\infty} c_{1+j} }\sum_{j=n_{0}}^{\infty}(j+1)e^{-j\epsilon_{t}}}\]
\[\stackrel{\text{Cor}.\,\ref{cora}}{=} \limsup_{t\to\infty} \frac{1+ e^{t\sum_{j=1}^{\infty} c_{1+j} +t\epsilon}(\epsilon_{t})^{-1}}{
1+e^{t\sum_{j=1}^{\infty} c_{1+j} }(\epsilon_{t})^{-2}}\]
\[= \limsup_{t\to\infty} \frac{1+ e^{t\left(\sum_{j=1}^{\infty} c_{1+j}\right) +t\epsilon -(tb+td+\left(t\sum_{j=1}^{\infty} c_{1+j}\right) +\psi(t))}}{
1+e^{t\left(\sum_{j=1}^{\infty} c_{1+j}\right) -2(tb+td+\left(t\sum_{j=1}^{\infty} c_{1+j}\right) +\psi(t)) }}\]
\[ \stackrel{b,\,d<0}{=}\limsup_{t\to\infty} \frac{e^{t\epsilon -(tb+td + \psi(t))}}{
e^{-t\left(\sum_{j=1}^{\infty} c_{1+j}\right) -2(tb+td +\psi(t)) }}\]
\[= \limsup_{t\to\infty} e^{t\epsilon + (tb+td + \psi(t))+t\left(\sum_{j=1}^{\infty} c_{1+j}\right)}.\]
%

For $\epsilon$ small this expression converges to
zero. Now, we return to $(\ref{eq3})$ and
we get  that it's only necessary to show that
\[ \limsup_{t\to\infty}\left((1 + \sum_{j=1}^{\infty}(j+1)e^{t(a_{2}+...+a_{1+j})-j\epsilon_{t}}) e^{t\epsilon + (tb+td + \psi(t))+t\left(\sum_{j=1}^{\infty} c_{1+j}\right)}\right) \leq 0.\]

If,
\[\limsup_{t\to\infty}\sum_{j=1}^{\infty}(j+1)e^{t(a_{2}+...+a_{1+j})-j\epsilon_{t}} <\infty,\]
then, the claim follows at once.
\newline
Suppose now that
\[\limsup_{t\to\infty}\sum_{j=1}^{\infty}(j+1)e^{t(a_{2}+...+a_{1+j})-j\epsilon_{t}} =\infty.\]

From an analogous argument, as used above in the second term analysis, we conclude that
\[ \limsup_{t\to\infty}(1 + \sum_{j=1}^{\infty}(j+1)e^{t(a_{2}+...+a_{1+j})-j\epsilon_{t}})\]
\[\leq \limsup_{t\to\infty} e^{t\left(\sum_{j=1}^{\infty}a_{1+j}\right) +t\epsilon -2(tb+td+t\left(\sum_{j=1}^{\infty}c_{1+j}\right) +\psi(t))} \]
So, we get
\[ \limsup_{t\to\infty}\left((1 + \sum_{j=1}^{\infty}(j+1)e^{t(a_{2}+...+a_{1+j})-j\epsilon_{t}}) e^{t\epsilon + (tb+td + \psi(t))+t\left(\sum_{j=1}^{\infty} c_{1+j}\right)}\right)\]
\[\leq \limsup_{t\to\infty}\left( e^{t\left(\sum_{j=1}^{\infty}a_{1+j}\right) +t\epsilon -2(tb+td+t\left(\sum_{j=1}^{\infty}c_{1+j}\right) +\psi(t))}e^{t\epsilon + (tb+td + \psi(t))+t\left(\sum_{j=1}^{\infty} c_{1+j}\right)}\right)\]
\[= \limsup_{t\to\infty} e^{t\left( (\sum_{j=1}^{\infty}a_{1+j}) -(b+d+(\sum_{j=1}^{\infty}c_{1+j}))\right) +2t\epsilon -\psi(t)} .\]
By hypothesis,  $(\sum_{j=1}^{\infty}a_{1+j}) -(b+d+(\sum_{j=1}^{\infty}c_{1+j}))<0$, so taking $\epsilon$ sufficiently small, and using the fact that $\psi(t)/t \to 0$, we get that $\frac{\mu_{t}([0])}{\mu_{t}([1])} \to 0$, exponentially fast.

\qed

\begin{example}

Assume $a=0,c=0$, $a_p=- 2^{2-p}$, $c_p= - 3^{2-p}$, $b=-2,d=-\frac{3}{2}$, $b_1=d_1<0,$
$b_p=a_2 +.. + a_p$, $d_p=c_2 +.. + c_p, \forall p\geq 2$.

 Moreover, $\beta(f)=0$, and, $P(tf)$ satisfies, for all $t$, the equation
$$
1=e^{t d_1- P(tf)}+ \sum_{j=1}^{\infty}e^{t(d_{1+j}+ b_{1+j}) - (j+1)P(tf)}= $$
$$
e^{t d_1- P(tf)}+ \sum_{j=1}^{\infty}e^{t( \frac{1-3^{-j}}{1-3^{-1}} + \frac{1-2^{-j}}{1-2^{-1}}) - (j+1)P(tf)}. $$

In this case, it is easy (Lemma \ref{epsilont}) to see that $A=(b+d)= -2 - 3/2= -7/2$.

Moreover,
$$V(1^\infty)= b-d=-1/2, V(01..)= b, V(0^p1..)= b -(a_2+...+a_p), V(0^\infty)=0,$$
 $$V(10..)= b, V(1^p0..)= b -(c_2+...+c_p),$$
and, $\beta_{\infty,t}= e^{t (b-d)}.$

Such $V$ is calibrated.

For a fixed $t$ and $q\geq 2$ we have
$$ \frac{\alpha_{q,t}}{\beta_{q,t}}=\frac{e^{t\,
(c_2+...+c_q)}  }{e^{t\,
(a_2+...+a_q)}  }$$

It is easy to see that $S_0(t)=S_1(t)$, for all $t$, therefore, $\mu_t([0])=1/2=\mu_t([1]).$ We also have $\mu_t([01])=\mu_t([10]).$

For fixed $n\geq 2$, we have

$$\frac{\mu_{t}([0^{n}1])}{\mu_{t}([1^{n}0])}=\frac{e^{t(a_{2}+...+a_{n}) +\log(\alpha_{n,t}) - \log(\alpha_{1,t})}}{e^{t(c_{2}+...+c_{n}) +\log(\beta_{n,t}) - \log(\beta_{1,t})} }=$$
$$e^{ t(\log(\beta_{1,t}) - \log(\alpha_{1,t}) ) }=1.$$

Therefore, $\mu_t \to \frac{1}{2} \delta_{0^\infty}+  \frac{1}{2} \delta_{1^\infty},$ as $t \to \infty$.

In this case we have a selection of a non-ergodic probability.

Note that
$$\lim_{t \to \infty} \frac{1}{t} \log(\alpha_{j,t})= V (0^j1..)=b -(a_2+...+a_j) ,$$
and
$$\lim_{t \to \infty} \frac{1}{t} \log(\alpha_{1,t})= V (01..)=b  .$$

Now, we would like to estimate the limit

$$\lim_{t \to \infty} \frac{1}{t} \log ( \mu(K)),$$

for a certain class of cylinder sets $K$. We point out that \cite{BLT} consider only the case where the maximizing probability is unique (see also \cite{LM} \cite{M} for a different setting which do not require uniqueness).

 In order to do that we need first to estimate
 $$\lim_{t \to \infty} \frac{1}{t} \log ( S_0 (t)),$$
where, by Proposition \ref{quo}

$$
S_{0}(t):= \frac{ e^{td_{1}} + \sum_{j=1}^{\infty}(j+1)e^{td_{1+j} +t(a_{2}+...+a_{1+j})-jP(tf)}}{e^{td_{1}}+\sum_{j=1}^{\infty}e^{td_{1+j} +t(a_{2}+...+a_{1+j})-jP(tf)}}.$$

We claim that
$$\lim_{t \to \infty} \frac{1}{t} \log ( S_0 (t))= -A =7/2.$$

Indeed, by lemma \ref{e} we get that $\frac{1}{t} \log$ applied to the denominator of $S_0(t)$ goes to
$$ max\{  sup_{j \geq 0}(d_{1+j} + a_{2} + ... + a_{1+j} ),  d + \sum_{j\geq
2} a_{j} - A\}.$$

In our case we can easily get that the above expression is $ d + \sum_{j\geq
2} a_{j} - A=-3/2\, -2\, + 7/2\,=\,0.$

Moreover, by Lemma \ref{cora}, $\frac{1}{t} \log$ applied to the numerator of $S_0(t)$ goes to
$$ max\{  sup_{j\geq 0}(d_{1+j} + a_{2} + ... + a_{1+j} ),  d + \sum_{j\geq 2} a_{j} - 2A\}=  d + \sum_{j\geq 2} a_{j} - 2A=7/2.$$
This shows the claim.

Note that in our case $\mu_t[01]= \frac{1}{2 S_0(t)}.$

Therefore,
$$\lim_{t \to \infty} \frac{1}{t} \log \mu_t[01]= -7/2.$$

In the same way
$$\lim_{t \to \infty} \frac{1}{t} \log \mu_t[10]= -7/2.$$

In a similar way as before, for $j\geq 2$
$$\lim_{t \to \infty} \frac{1}{t} \log \mu_t[0^j1]= \lim_{t \to \infty} \frac{1}{t} \log \mu_{t}([01])e^{-(j-1)P(tf)+t(a_{2}+...+a_{j}) +\log(\alpha_{j,t}) - \log(\alpha_{1,t})}=
$$
$$\lim_{t \to \infty} \frac{1}{t} \log \mu_{t}([01])e^{-(j-1)P(tf)}=-7/2- \lim_{t \to \infty}  (j-1)\, \frac{P(tf)}{t}=-7/2 .$$

In the same way
$$\lim_{t \to \infty} \frac{1}{t} \log \mu_t[1^j0]=-7/2.$$

The final conclusion is that, for the present example, we are able to get the above limit for a certain class of cylinders  $K$. The estimates are all explicit.

One can show that the limit  $\lim_{t \to \infty} \frac{1}{t} \log \mu_t(K)$ exists, and it is also true a  Large Deviation Principle with the usual deviation function $I$ \cite{BLT}.

\end{example}


\begin{thebibliography}{99}





\bibitem{BLL} A. Baraviera, R. Leplaideur and A. O. Lopes, Selection
of ground states in the zero-temperature limit for a one-parameter family of
potentials, SIAM Journal on Applied Dynamical Systems, pre-print ARXIV






 \bibitem{BLT} A. Baraviera, A. O. Lopes and Ph. Thieullen,
A Large Deviation Principle for Gibbs states of Holder potentials: the zero temperature case. \emph{ Stoch. and  Dyn.}\, (6), 77-96, (2006).



\bibitem{BCLMS}
A. T.  Baraviera, L. M. Cioletti,  A. O. Lopes,
J. Mohr  and R. R. Souza,
On the general $XY$ Model: positive and zero temperature, selection and non-selection, preprint UFRGS (2011)






\bibitem{Bousch1}
T. Bousch,  Le poisson n'a pas d'ar\^etes, \emph{Annales de
l'Institut Henri Poincar\'e, Probabilit\'es et Statistiques},
Vol 36, 489-508 (2000).


\bibitem{Bousch-walters}
T. Bousch,
\newblock La condition de Walters.
\newblock {\em Ann. Sci. ENS}, 34, (2001), 287â311.





\bibitem{Bremont}
J.~Br{\'e}mont,
\newblock Gibbs measures at temperature zero.
\newblock {\em Nonlinearity}, 16(2): 419--426, 2003.







\bibitem{chazottes-cv}
J.R. Chazottes and M. Hochman,
On the zero-temperature limit of Gibbs states, \emph{Commun. Math. Phys}., Volume 297, N. 1, 2010


\bibitem{Chazottes-Gambaudo-Ulgade}
J.R. Chazottes, J.M. Gambaudo and E. Ulgade,
\newblock Zero-temperature limit of one dimensional Gibbs states via renormalization: the case of locally constant potentials.
\newblock to appear in \emph{ Erg. Theo. and Dyn. Sys.}.



\bibitem{CoG} J. P. Conze and Y. Guivarc'h, \emph{Croissance
des sommes ergodiques et principe variationnel}, manuscript, circa 1993.


\bibitem{CLT}
G. Contreras, A. O. Lopes and Ph. Thieullen, Lyapunov minimizing
measures for expanding maps of the circle, \emph{Ergodic Theory
and Dynamical Systems} \textbf{21} (2001), 1379-1409.

\bibitem{FL} A. Fisher and A. O. Lopes, Exact bounds for the polynomial decay of correlation, 1/f noise
and the CLT for the equilibrium state of a non-H\"older potential,
Nonlinearity, Vol 14, Number 5, pp 1071-1104 (2001).



\bibitem{GL1}
E. Garibaldi, A. O. Lopes, On the Aubry-Mather theory for symbolic dynamics,
\emph{Ergodic Theory and Dynamical Systems} \textbf{28} (2008), 791-815.



\bibitem{Hof}
F. Hofbauer, Examples for the non-uniqueness
of the Gibbs states,\emph{  Trans. AMS }, \, (228), 133-141,
(1977)

\bibitem{Jenkinson}
O. Jenkinson, Ergodic optimization, \emph{Discrete and Continuous Dynamical Systems, Series A}
\textbf{15} (2006), 197-224.


 \bibitem{Kel} G. Keller, Gibbs States in Ergodic Theory, Cambrige Press, 1998.


\bibitem{leplaideur-max}
R.~Leplaideur,
A dynamical proof for the convergence of {G}ibbs measures at
  temperature zero.
\emph{ Nonlinearity}, 18(6):2847--2880, 2005.
\bibitem{Le1}
R. Leplaideur, Local product structure for
Gibbs states. \emph{ Trans. AMS.} 352 (2000), no. 4, 1889-1912.

\bibitem{L1} A. O. Lopes, The Zeta Function, Non-Differentiability of Pressure
and the Critical Exponent of Transition,  \emph{ Advances in
Mathematics}, Vol. 101, pp. 133-167, (1993).


\bibitem{LM} A. O. Lopes and J. Mengue, Zeta measures  and Thermodynamic Formalism for temperature zero, \emph{Bulletin of the Brazilian Mathematical Society} 41 (3) pp 449-480 (2010)



\bibitem{Man} R. Mane, Ergodic Theory and Differentiable Dynamics, Springer Verlag (1987)




\bibitem{M} J. Mengue, Zeta-medidas e princ\'ipio dos grandes desvios, PhD thesis, UFRGS (2010)\,\,
http://hdl.handle.net/10183/26002


\bibitem{Mo}
I. D. Morris,  A sufficient condition for the subordination principle in ergodic optimization, \emph{ Bull. Lond. Math. Soc.} \textbf{39}, no. 2, (2007), 214-220.



\bibitem{PP}
W. Parry and M. Pollicott,
Zeta functions and the periodic orbit structure of hyperbolic dynamics,
\emph{Ast\'erisque} \textbf{187-188} (1990).


\bibitem{VE} A. van Enter and W. Ruszel, Chaotic temperature dependence at zero temperature. \emph{J. Stat. Phys.} 127 (2007), no. 3, 567-573.



\bibitem{W} P. Walters, A natural
space of functions for the Ruelle operator theorem,
\emph{ Ergodic Theory and Dynamical Systems.} 27,
1323-1348 (2007).



	

\end{thebibliography}
\end{document}